\documentclass[11pt]{amsart}
\usepackage[sorted-cites]{amsrefs}
\usepackage{verbatim}
\usepackage{color}
\usepackage{amsmath,amsthm,amscd}
\usepackage{tikz}

\newtheorem{thm}{Theorem}

\newtheorem{lem}{Lemma}
\newtheorem{cor}{Corollary}

\newtheorem{example}{Example}
\theoremstyle{remark}
\newtheorem{rem}{Remark}

\makeatletter

\begin{document}
\global\long\def\tr{\operatorname{tr}}

\global\long\def\proj{\operatorname{proj}}

\global\long\def\ric{\operatorname{Ric}}

\title[]{Geometric properties of self-shrinkers in cylinder shrinking Ricci solitons}

\author{Matheus Vieira and Detang Zhou}
\thanks{The work was partially supported by CNPq and Faperj of Brazil.}
\address{Departamento de Matem\'atica\\Universidade Federal do Esp\'\i rito Santo \\Vit\'oria, ES 29075-910, Brazil.} \email{matheus.vieira@ufes.br
}
\address{Instituto de Matem\'atica e Estat\'\i stica\\ Universidade Federal Fluminense \\Niter\'oi, RJ 24020-140, Brazil.} \email{zhou@impa.br}

\maketitle
\begin{abstract}
In this paper we prove some spectral properties of the drifted Laplacian
of self-shrinkers properly immersed in gradient shrinking Ricci solitons. Then we use these results to prove some geometric properties of self-shrinkers. For example, we describe a collection of domains in the ambient space that cannot contain self-shrinkers.
\end{abstract}

\section{Introduction}

It is an important problem in geometry to study the singularities
of the mean curvature flow. Self-shrinkers play a key role in the study of type I
singularities of the flow. Recall that a self-shrinker is a submanifold
of $\mathbb{R}^{n+1}$ satisfying 
\[
\overrightarrow{H}+\frac{1}{2}x^{\perp}=0,
\]
where $\overrightarrow{H}$ is the mean curvature vector and $x^{\perp}$
is the projection of the position vector to the normal bundle. Many important results in this topic have been done since  the publication of
Colding and Minicozzi \cite{CM}. 

Self-shrinkers are often compared with minimal surfaces. Hoffman and
Meeks \cite{HM} proved that a nonplanar minimal surface properly
immersed in $\mathbb{R}^{3}$ cannot lie in a open halfspace. In
higher dimensions this result fails. The main difference is that a
catenoid in $\mathbb{R}^{3}$ cannot lie in a open halfspace but
a catenoid in $\mathbb{R}^{n+1}$ lies between two parallel hyperplanes
for $n\geq3$ (see \cite{MR2515414} for properties of the higher dimensional catenoid). Recently, Cavalcante and Espinar \cite{CE} proved a halfspace theorem for self-shrinkers properly immersed in $\mathbb{R}^{n+1}$. The proof uses the idea of Hoffman and Meeks and a catenoid type self-shrinker discovered by Kleene and Moller \cite{KM}. Note that there are many examples of compact self-shrinkers and noncompact proper self-shrinkers.

In this paper we show that  strong halfspace type theorems for self-shrinkers are quite different from that of minimal surfaces and it is natural to use eigenvalues and eigenfunctions of the corresponding drifted Laplacian to prove them. We describe a collection of domains in the ambient space that cannot contain proper self-shrinkers. We also show some other geometric properties of self-shrinkers. Moreover, using spectral properties of the drifted Laplacian we can give short proofs of these results. Cheng and Zhou \cite{CZ} proved that for a self-shrinker in $\mathbb{R}^{n+1}$
the conditions of proper immersion, Euclidean volume growth, polynomial
volume growth and finite weighted volume are equivalent to each other.
They also proved that the drifted Laplacian of a self-shrinker properly
immersed in $\mathbb{R}^{n+1}$ has discrete spectrum \cite{CZ13}.
Thus any eigenfunction of the drifted Laplacian with nonzero eigenvalue
is orthogonal to the constant function with respect to the $L_{f}^{2}$
inner product. The spectrum of the drifted Laplacian contains the
eigenvalues $\left\{ 0,\frac{1}{2},1\right\} $ with corresponding
eigenfunctions $\left\{ 1;x_{1},\dots,x_{n+1};2n-\left|x\right|^{2}\right\} $.
Applying the idea above to these eigenfunctions we easily prove the
following result.
\begin{thm}
\label{thm:selfshrinkereuclidean1}Any self-shrinker hypersurface properly immersed in $\mathbb{R}^{n+1}$ intersects all members
of the collection $\mathcal{C}$ given by 
\[
\mathcal{C}:=\left\{ S^{n}\left(tp,\sqrt{2n+t^{2}}\right):p\textrm{ is a unit vector in }\mathbb{R}^{n+1}\textrm{ and }0\leq t\leq\infty\right\} .
\]
Moreover, if the self-shrinker lies in the closed set $\overline{B^{n+1}\left(tp,\sqrt{2n+t^{2}}\right)}$,
then either it is the sphere $S^{n}\left(o,\sqrt{2n}\right)$ with $t=0$,
or it is the hyperplane $\left\langle x,p\right\rangle =0$ with $t=\infty$.\end{thm}
\begin{rem}
As $t\to\infty$ the sphere $\ensuremath{S^{n}\left(tp,\sqrt{2n+t^{2}}\right)}$
converges to the hyperplane $\left\langle x,p\right\rangle =0$ and
the closed ball $\overline{B^{n+1}\left(tp,\sqrt{2n+t^{2}}\right)}$
converges to the closed halfspace $\left\langle x,p\right\rangle \geq0$.
For any vector $v$ in $\mathbb{R}^{n+1}$ the collection $\mathcal{C}$
contains the sphere $S^{n}\left(v,\sqrt{2n+\left|v\right|^{2}}\right)$
as well the hyperplane $\left\langle x,v\right\rangle =0$.
\end{rem}
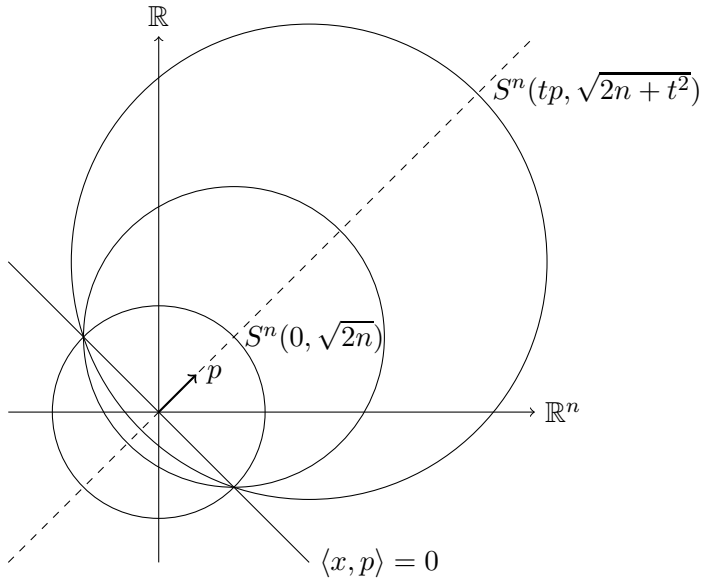
\begin{tikzpicture}
\draw[->] (-2,0) -- (5,0) node[right] {$\mathbb{R}^{n}$};
\draw[->] (0,-2) -- (0,5) node[above] {$\mathbb{R}$};
\draw [dashed] (-2,-2) -- (5,5);
\draw[->, thick] (0,0) -- (0.5,0.5) node[right] {$p$};
\draw (0,0) circle [radius=sqrt(2)];
\node [right] at (1,1) {$S^{n}(0,\sqrt{2n})$};
\draw (1,1) circle [radius=sqrt(4)];
\draw (2,2) circle [radius=sqrt(10)];
\node [right] at (4.3,4.3) {$S^{n}(tp,\sqrt{2n+t^2})$};
\draw (-2,2) -- (2,-2) node[right] {$\langle x,p \rangle =0$};
\end{tikzpicture}

Among all self-shrinkers, spheres and  cylinders deserve special attention since they are generic singularities of the mean curvature flow \cite{CM}. Moreover, Colding and Minicozzi \cite{CM2} proved that  the blowup at each generic singularity of a mean curvature flow is unique; that is,  it does not depend on the sequence of rescalings.

Extending  Theorem \ref{thm:selfshrinkereuclidean1} we prove the following result.
\begin{cor}
\label{cor:selfshrinkereuclidean2}Any self-shrinker hypersurface properly immersed in $\mathbb{R}^{n+1}$ satisfies the following
geometric properties:

(a) If the Gaussian image of the self-shrinker lies in a closed semisphere, then
it is a hyperplane.

(b) The self-shrinker cannot lie inside the closed product
\[
\overline{B^{k+1}\left(v,\sqrt{2k+\left|v\right|^{2}}\right)}\times\mathbb{R}^{n-k}
\]
for any vector $v$ in $\mathbb{R}^{k+1}$, unless it is the product $S^{k}\left(o,\sqrt{2k}\right)\times\mathbb{R}^{n-k}$ and $v=0$.

(c) The self-shrinker cannot lie outside the closed product
\[
\overline{B^{k+1}\left(v,\sqrt{2\left(k+1\right)+\left|v\right|^{2}}\right)}\times\mathbb{R}^{n-k}
\]
for any vector $v$ in $\mathbb{R}^{k+1}$.

(d) There is no hyperplane passing through the origin separating the self-shrinker into two parts with one part lying inside the closed ball $\overline{B^{n+1}\left(o,\sqrt{2n}\right)}$
and the other part lying outside the open ball $B^{n+1}\left(o,\sqrt{2n}\right)$,
unless it is the sphere $S^{n}\left(o,\sqrt{2n}\right)$.\end{cor}
\begin{rem}
Item (a) was proved by Ding, Xin and Yang \cite{DXY} by a different
method.
\end{rem}
It is also an interesting problem in geometry to study the mean curvature flow  in more general ambient spaces. In the nonflat case Huisken \cite{Hu87} studied the asymptotic
behavior of certain hypersurfaces evolving under the mean curvature flow immersed in the sphere. Recently, Sheng and Yu \cite{SY} generalized Huisken results to hypersurfaces evolving under the mean curvature flow immersed in Riemannian manifolds evolving under a normalized Ricci flow.

The study of the singularities of the mean curvature flow is related
to $f$-minimal hypersurfaces and submanifolds. It is easy to see
that a submanifold of $\mathbb{R}^{n+1}$ is a self-shrinker if and
only if it is $f$-minimal for $f\left(x\right)=\left|x\right|^{2}/4$.
Recently, Yamamoto \cite{Y} proved that certain type I singularities
of the Ricci-mean curvature flow are $f$-minimal submanifolds. More
precisely, given a gradient shrinking Ricci soliton $\left(\overline{M},g,f\right)$
and $T>0$ consider the canonical solution of the Ricci flow in the
interval $-\infty<t<T$, and given a compact manifold $M$ consider
a family of immersions $F:M\times\left[0,T\right)\to\overline{M}$
evolving under the mean curvature flow. Assuming that the coupled
flow has a certain type I singularity at time $T$, Yamamoto proved
that rescaled flow converges smoothly to a $f$-minimal submanifold
of $\overline{M}$. Thus, a $f$-minimal submanifold of a gradient
shrinking Ricci soliton appears as a singularity of the Ricci-mean
curvature flow in the same way that a self-shrinker submanifold of
$\mathbb{R}^{n+1}$ appears as a singularity of the mean curvature flow. In this paper $f$-minimal submanifolds of gradient shrinking Ricci solitons are also called self-shrinker submanifolds. Note that gradient shrinking Ricci solitons are also certain singularities
of the Ricci flow (see \cite{P} and the recent paper \cite{EMT}).

Recently, Cheng, Mejia and Zhou extended the Choi-Wang estimate and
the Colding-Minicozzi compactness theorem to $f$-minimal surfaces
in \cite{CMZ} and \cite{CMZ3} respectively. In \cite{CMZ2}
and \cite{CZ13-2} they classified index-one self-shrinker hypersurfaces
properly immersed in the gradient shrinking Ricci soliton $\mathbb{R}\times S^{n}\left(o,\sqrt{2\left(n-1\right)}\right)$.

In this paper we focus in the case when the ambient space is a round cylinder gradient shrinking Ricci soliton $\left(\mathbb{R}^{p}\times S_{\sqrt{2\left(k-1\right)}}^{k},g,f\right)$ with product metric $g$ and $f\left(x,y\right)=\frac{\left|x\right|^{2}}{4}$. Round cylinders are important in the study of the mean curvature flow. Colding and Minicozzi \cite{CM} proved that the generic singularities of the mean curvature flow are generalized round cylinders. They are also important in the study of the Ricci flow. Even though it is not known at this moment whether the generic singularities of the Ricci flow are generalized round cylinders, it is known that conformally flat gradient shrinking Ricci solitons are quotients of $S^n$, $\mathbb{R}^n$ and  $\mathbb{R}\times S_{\sqrt{2\left(n-2\right)}}^{n-1}$. In fact, the works of Fern\'andez-L\'opez and Garc\'ia-R\'io \cite{FG2}  and Munteanu-Sesum \cite{MS} led to the classification of gradient shrinking solitons of dimension $n\geq 4$ and harmonic Weyl tensor (vanishing Cotton tensor): they are either  Einstein manifolds or a finite quotient of the product $\mathbb{R}^{n-k} \times N^k$ of the Gaussian shrinking soliton $\mathbb R^{n-k}$ with a positive Einstein manifold $N^k$ where $2\leq k<n$. We would like to point out that some results of this paper holds for more general gradient Ricci solitons. For example, the first part of Theorem  \ref{thm:spectrumss} holds when the $S^k$ factor is replaced by any closed Einstein manifold with Einstein constant $\frac12$ (see Lemma \ref{lem:discrete} for a even more general result).

The idea of using the spectral properties of the drifted Laplacian can
be applied to submanifolds of gradient shrinking Ricci solitons. Cheng, Mejia and
Zhou \cite{CMZ3} proved that for a self-shrinker submanifold of $\mathbb{R}^{p}\times S_{\sqrt{2\left(k-1\right)}}^{k}$
the conditions of proper immersion, polynomial volume growth and finite
weighted volume are equivalent to each other. By this fact, Lemma
\ref{lem:equations} and Lemma \ref{lem:discrete} we have the following
result.
\begin{thm}
\label{thm:spectrumss}Given a self-shrinker submanifold properly
immersed in the gradient shrinking Ricci soliton $\mathbb{R}^{p}\times S_{\sqrt{2\left(k-1\right)}}^{k}$,
the drifted Laplacian of the self-shrinker has discrete spectrum. Moreover, the
spectrum contains the eigenvalues $0$ and $\frac{1}{2}$ with corresponding
eigenfunctions $1$ and $x_{1},\cdots,x_{p}$.\end{thm}
\begin{rem}
For $k=0$ the corresponding result for self-shrinker submanifolds properly immersed in $\mathbb{R}^{p}$ was proved in \cite{CZ13}. The proof uses a logarithmic Sobolev inequality obtained by Ecker \cite{E}. This inequality holds when the ambient space is a gradient shrinking Ricci soliton with bounded geometry \cite{W}. Here we give an alternative proof of the result in \cite{CZ13}.

Applying spectral properties of the drifted Laplacian we prove the
following result.\end{rem}
\begin{thm}
\label{thm:selfshrinkercylinder}Any self-shrinker hypersurface properly immersed in the gradient shrinking Ricci soliton $\mathbb{R}^{n+1-k}\times S_{\sqrt{2\left(k-1\right)}}^{k}$
satisfies the following geometric properties:

(a) If the self-shrinker lies inside the closed product $\overline{B_{\sqrt{2\left(n-k\right)}}^{n+1-k}}\times S_{\sqrt{2\left(k-1\right)}}^{k}$,
then it is the product $S_{\sqrt{2\left(n-k\right)}}^{n-k}\times S_{\sqrt{2\left(k-1\right)}}^{k}$.

(b) The self-shrinker cannot lie outside the closed product
$\overline{B_{\sqrt{2\left(n+1-k\right)}}^{n+1-k}}\times S_{\sqrt{2\left(k-1\right)}}^{k}$.

(c) If the self-shrinker lies in the product of a closed halfspace in $\mathbb{R}^{n+1-k}$
passing through the origin with $S_{\sqrt{2\left(k-1\right)}}^{k}$,
then it is the product of the separating hyperplane with $S_{\sqrt{2\left(k-1\right)}}^{k}$.
\end{thm}
The paper is organized as follows. In Section 2 we present the basic conventions and notations. In Section 3 we prove some spectral properties of the drifted Laplacian. In Section 4 we prove Theorem \ref{thm:selfshrinkereuclidean1}, Corollary \ref{cor:selfshrinkereuclidean2} and Theorem \ref{thm:selfshrinkercylinder}.

\section{Conventions and notations}

In this section we recall some definitions and present the conventions and notations of the paper.

Recall that a smooth metric measure space is a triple $\left(M,g,f\right)$ of a Riemannian manifold $\left(M,g\right)$ and a smooth function $f$ on $M$. The weighted volume is the measure $e^{-f}dvol$. The drifted Laplacian is defined by 
\[
\Delta_{f}=\Delta-\left\langle \nabla f,\nabla\cdot\right\rangle .
\]
The $L_{f}^{2}\left(M\right)$ inner product of functions $u$ and $v$ on $M$ is defined by 
\[
\left\langle u,v\right\rangle _{L_{f}^{2}\left(M\right)}=\int_{M}uve^{-f}.
\]
It is well known that the drifted Laplacian is a densely defined self-adjoint
operator in $L_{f}^{2}\left(M\right)$, i.e. for smooth functions
$u$ and $v$ on $M$ with compact we have
\[
\int_{M}\Delta_{f}u\cdot ve^{-f}=-\int_{M}\left\langle \nabla u,\nabla v\right\rangle e^{-f}.
\]
The Bakry-\'Emery-Ricci curvature is defined by 
\[
\ric_{f}=\ric+\nabla\nabla f.
\]
A gradient Ricci soliton is a smooth metric measure space with
\[
\ric_{f}=cg
\]
for some constant $c$. It is shrinking, steady or expanding if this constant is positive, zero or negative respectively.

Recall that for a submanifold $M$ of a smooth metric measure space $\left(\overline{M},g,f\right)$ the second fundamental form is defined by
\[
A\left(X,Y\right)=\left(\overline{\nabla}_{X}Y\right)^{\perp}.
\]
Here we consider the Riemannian connection $\overline{\nabla}$ of the ambient space $\overline{M}$ and the projection $\perp$ to the normal bundle. The mean curvature vector is defined by 
\[
\overrightarrow{H}=\tr A.
\]
The submanifold is called $f$-minimal if
\[
\overrightarrow{H}+\left(\overline{\nabla}f\right)^{\perp}=0.
\]
It is well known that $M$ is $f$-minimal if and only if it is a
critical point of the weighted volume functional $vol_{f}\left(M\right)=\int_{M}e^{-f}$
with respect to compactly supported normal variations (e.g. \cite{CMZ3}). 
Note that $\left(M,g,f\right)$ is also a smooth metric measure space and drifted Laplacian of $M$ is given by
\[
\Delta_{f}=\Delta-\left\langle \nabla f,\nabla\cdot\right\rangle .
\]
Here we consider the Riemannian connection $\nabla$ and the Laplacian
$\Delta$ of the submanifold $M$.

In this paper by a self-shrinker submanifold we mean a $f$-minimal submanifold of gradient shrinking Ricci soliton. We usually denote the geometric quantities of the ambient space with a bar over them. The open ball $B^{k+1}\left(v,r\right)$ in $\mathbb{R}^{k+1}$ of radius $r$ and center $v$ has boundary $S^{k}\left(v,r\right)$ and closure $\overline{B^{k+1}\left(v,r\right)}$. The notation of closure for balls should not be confused with the notation of the ambient space. Sometimes we write $B_{r}^{k+1}=B^{k+1}\left(o,r\right)$ and $S_{r}^{k}=S^{k}\left(o,r\right)$.

\begin{example}
Consider the ambient space $\mathbb{R}^{p}$ with canonical metric
$g$ and the function $f\left(x\right)=\frac{\left|x\right|^{2}}{4}$.
The triple $\left(\mathbb{R}^{p},g,f\right)$ is a gradient shrinking
Ricci soliton with Bakry-\'Emery-Ricci curvature 
\[
\overline{\ric}_f=\frac{1}{2}g.
\]
A submanifold $M$ of $\mathbb{R}^{p}$ is a self-shrinker if and
only if
\[
\overrightarrow{H}+\frac{1}{2}x^{\perp}=0.
\]
The drifted Laplacian of $M$ is given by 
\[
\Delta_{f}=\Delta-\frac{1}{2}\left\langle x,\nabla\cdot\right\rangle .
\]
The operator $\Delta_{f}$ is equal to the operator $\mathcal{L}$
introduced by Colding and Minicozzi \cite{CM}.
\end{example}

\begin{example}
Consider the ambient space $\mathbb{R}^{p}\times S_{\sqrt{2\left(k-1\right)}}^{k}$
with product metric $g$ and the function $f\left(x,y\right)=\frac{\left|x\right|^{2}}{4}$.
Here $x$ is the position vector in $\mathbb{R}^{p}$ and $y$ is
the position vector in $\mathbb{R}^{k+1}$. The triple $\left(\mathbb{R}^{p}\times S_{\sqrt{2\left(k-1\right)}}^{k},g,f\right)$
is a gradient shrinking Ricci soliton with Bakry-\'Emery-Ricci curvature
\begin{align*}
\overline{\ric}_{f} & =\overline{\nabla}\overline{\nabla}f+\overline{\ric}\\
 & =\frac{1}{2}g_{\mathbb{R}^{p}}+\frac{1}{2}g_{S_{\sqrt{2\left(k-1\right)}}^{k}}\\
 & =\frac{1}{2}g.
\end{align*}
A submanifold $M$ of $\mathbb{R}^{p}\times S_{\sqrt{2\left(k-1\right)}}^{k}$
is a self-shrinker if and only if
\[
\overrightarrow{H}+\frac{1}{2}x^{\perp}=0.
\]
The drifted Laplacian of $M$ is given by 
\[
\Delta_{f}=\Delta-\frac{1}{2}\left\langle x,\nabla\cdot\right\rangle .
\]

\end{example}

\section{Spectral properties of the drifted Laplacian}

In this section, first we characterize self-shrinker submanifolds
of the gradient shrinking Ricci soliton $\mathbb{R}^{p}\times S_{\sqrt{2\left(k-1\right)}}^{k}$.
In particular, for $k=0$ a submanifold of $\mathbb{R}^{p}$ is a
self-shrinker if and only if the restrictions of the cartesian coordinates
of $\mathbb{R}^{p}$ are eigenfunctions of the drifted Laplacian of
the submanifold with eigenvalue $\frac{1}{2}$. Then we prove some
differential equations and differential inequalities for self-shrinker
submanifolds properly immersed in $\mathbb{R}^{p}\times S_{\sqrt{2\left(k-1\right)}}^{k}$.
These identities and inequalities are used in the proofs of the results
of next section. Then we prove that for self-shrinker submanifolds
properly immersed in certain gradient shrinking Ricci solitons the
spectrum of the drifted Laplacian of the submanifold is discrete.
This result can be applied to self-shrinker submanifolds properly immersed
in $\mathbb{R}^{p}\times S_{\sqrt{2\left(k-1\right)}}^{k}$. Finally,
we prove a useful lemma for the drifted Laplacian of $f$-parabolic
smooth metric measure spaces.
\begin{lem}
\label{lem:volume}For a complete self-shrinker submanifold $M$ of
the gradient shrinking Ricci soliton $\mathbb{R}^{p}\times S_{\sqrt{2\left(k-1\right)}}^{k}$
the conditions proper immersion, polynomial volume growth and finite
weighted volume (i.e. $\int_{M}e^{-f}<\infty$) are equivalent to
each other.\end{lem}
\begin{proof}
The conclusion follows from Corollary 1 in \cite{CMZ3} because $\mathbb{R}^{p}\times S_{\sqrt{2\left(k-1\right)}}^{k}$
is a gradient shrinking Ricci soliton and the function $f\left(x,y\right)=\frac{\left|x\right|^{2}}{4}$
is convex.\end{proof}
\begin{lem}
\label{lem:characterization}A submanifold $M$ of the gradient shrinking
Ricci soliton $\mathbb{R}^{p}\times S_{\sqrt{2\left(k-1\right)}}^{k}$
is a self-shrinker if and only if the cartesian coordinates $x_{1},\dots,x_{p}$
of $\mathbb{R}^{p}$ are eigenfunctions of drifted Laplacian of $M$
with eigenvalue $\frac{1}{2}$ i.e.
\[
\Delta_{f}x_{i}=-\frac{1}{2}x_{i},
\]
and the cartesian coordinates $y_{1},\dots,y_{k+1}$ of $\mathbb{R}^{k+1}$
are eigenfunctions of drifted Laplacian of $M$ with eigenvalue $\frac{1}{2\left(k-1\right)}\tr_{M}g_{S_{\sqrt{2\left(k-1\right)}}^{k}}$
i.e.
\[
\Delta_{f}y_{i}=-\left(\frac{1}{2\left(k-1\right)}\tr_{M}g_{S_{\sqrt{2\left(k-1\right)}}^{k}}\right)y_{i}.
\]
\end{lem}
\begin{proof}
By direct calculation for any function $u$ on $\mathbb{R}^{p}\times S_{\sqrt{2\left(k-1\right)}}^{k}$ we have
\[
\nabla\nabla u=\overline{\nabla}\overline{\nabla}u+\left\langle \overline{\nabla}u,A\left(\cdot,\cdot\right)\right\rangle .
\]
Here we consider the Riemannian connection $\nabla$ of $M$, the Riemannian connection $\overline{\nabla}$ of $\mathbb{R}^{p}\times S_{\sqrt{2\left(k-1\right)}}^{k}$ and the second fundamental form $A$ of $M$ with respect to  $\mathbb{R}^{p}\times S_{\sqrt{2\left(k-1\right)}}^{k}$.
Thus 
\begin{align}
\Delta_{f}u & =\tr_{M}\overline{\nabla}\overline{\nabla}u+\left\langle \overline{\nabla}u,\overrightarrow{H}\right\rangle -\left\langle \nabla f,\nabla u\right\rangle \nonumber \\
 & =\tr_{M}\overline{\nabla}\overline{\nabla}u+\left\langle \overline{\nabla}u,\overrightarrow{H}+\left(\overline{\nabla}f\right)^{\perp}\right\rangle -\left\langle \overline{\nabla}f,\overline{\nabla}u\right\rangle .\label{eq:fsubman}
\end{align}
The function $f\left(x,y\right)=\frac{\left|x\right|^{2}}{4}$
satisfies 
\[
\overline{\nabla}f=\frac{1}{2}x.
\]
By direct calculation for any function $u$ on $\mathbb{R}^{p}\times\mathbb{R}^{k+1}$ we have
\[
\overline{\nabla}\overline{\nabla}u=\widetilde{\nabla}\widetilde{\nabla}u+\left\langle \widetilde{\nabla}u,\overline{A}\left(\cdot,\cdot\right)\right\rangle .
\]
Here we consider the Riemannian connection $\widetilde{\nabla}$ of $\mathbb{R}^{p}\times\mathbb{R}^{k+1}$ and the second fundamental form $\overline{A}$ of $\mathbb{R}^{p}\times S_{\sqrt{2\left(k-1\right)}}^{k}$ with respect to $\mathbb{R}^{p}\times\mathbb{R}^{k+1}$, which is given by
\[
\overline{A}=-\frac{1}{2\left(k-1\right)}g_{S_{\sqrt{2\left(k-1\right)}}^{k}}\left(\cdot,\cdot\right)y.
\]

Take the canonical basis $e_{1},\dots,e_{p+k+1}$ of $\mathbb{R}^{p}\times\mathbb{R}^{k+1}$. Since $\widetilde{\nabla}y_{i}=e_{p+i}$ and $\widetilde{\nabla}\widetilde{\nabla}y_{i}=0$, by the two identities above we have
\begin{equation}
\overline{\nabla}\overline{\nabla}y_{i}=-\frac{1}{2\left(k-1\right)}y_{i}g_{S_{\sqrt{2\left(k-1\right)}}^{k}}.\label{eq:coordinatesphere}
\end{equation}

First assume that $M$ is a self-shrinker. Since $\overline{\nabla}x_{i}=e_{i}$, $\overline{\nabla}\overline{\nabla}x_{i}=0$ and $\left\langle \overline{\nabla}f,\overline{\nabla}x_{i}\right\rangle =\frac{1}{2}x_{i}$, by Equation (\ref{eq:fsubman}) we have
\[
\Delta_{f}x_{i}=-\frac{1}{2}x_{i}.
\]
Since $\overline{\nabla}y_{i}=\proj_{TS_{\sqrt{2\left(k-1\right)}}^{k}}e_{p+i}$
and $\left\langle \overline{\nabla}f,\overline{\nabla}y_{i}\right\rangle =0$, by Equation (\ref{eq:fsubman}) and Equation (\ref{eq:coordinatesphere}) we have
\[
\Delta_{f}y_{i}=-\left(\frac{1}{2\left(k-1\right)}\tr_{M}g_{S_{\sqrt{2\left(k-1\right)}}^{k}}\right)y_{i}.
\]

Now assume that $x_{i}$ and $y_{i}$ are eigenfunctions of the drifted
Laplacian of $M$. Since $\Delta_{f}x_{i}=-\frac{1}{2}x_{i}$, taking
$u=x_{i}$ in Equation (\ref{eq:fsubman}) we obtain
\[
\left\langle \overrightarrow{H}+\left(\overline{\nabla}f\right)^{\perp},e_{i}\right\rangle =0.
\]
Since this identity holds for $i=1,\dots,p$ it follows that 
\[
\proj_{\mathbb{R}^{p}}\left(\overrightarrow{H}+\left(\overline{\nabla}f\right)^{\perp}\right)=0.
\]
Since 
\[
\Delta_{f}y_{i}=-\left(\frac{1}{2\left(k-1\right)}\tr_{M}g_{S_{\sqrt{2\left(k-1\right)}}^{k}}\right)y_{i},
\]
taking $u=y_{i}$ in Equation (\ref{eq:fsubman}) and using Equation
(\ref{eq:coordinatesphere}) we obtain
\[
\left\langle \overrightarrow{H}+\left(\overline{\nabla}f\right)^{\perp},\proj_{TS_{\sqrt{2\left(k-1\right)}}^{k}}e_{p+i}\right\rangle =0.
\]
Since this identity holds for $i=1,\dots,k+1$ it follows that 
\[
\proj_{TS_{\sqrt{2\left(k-1\right)}}^{k}}\left(\overrightarrow{H}+\left(\overline{\nabla}f\right)^{\perp}\right)=0.
\]
Therefore 
\[
\overrightarrow{H}+\left(\overline{\nabla}f\right)^{\perp}=0.
\]
\end{proof}
\begin{cor}
A submanifold $M$ of $\mathbb{R}^{p}$ is a self-shrinker if and
only if the cartesian coordinates $x_{1},\dots,x_{p}$ of $\mathbb{R}^{p}$
are eigenfunctions of the drifted Laplacian of $M$ with eigenvalue
$\frac{1}{2}$.\end{cor}
\begin{lem}
\label{lem:equations}For a self-shrinker submanifold $M$ properly
immersed in the gradient shrinking Ricci soliton $\mathbb{R}^{p}\times S_{\sqrt{2\left(k-1\right)}}^{k}$
the following assertions hold:

(a) $\Delta_{f}\left\langle x,v\right\rangle =-\frac{1}{2}\left\langle x,v\right\rangle $
for any $v$ in $\mathbb{R}^{p}\times\mathbb{R}^{k+1}$;

(b) $\Delta_{f}\left|x\right|^{2}=-\left|x\right|^{2}+2p-2\sum_{\alpha}\left|\proj_{\mathbb{R}^{p}}\nu_{\alpha}\right|^{2}$;

(c) $\Delta_{f}\left(\sum_{i=1}^{j}x_{i}^{2}\right)\geq-\sum_{i=1}^{j}x_{i}^{2}+2\left(j-\dim TM^{\perp}\right)$;

(d) $\Delta_{f}\left(\sum_{i=1}^{j}x_{i}^{2}\right)\leq-\sum_{i=1}^{j}x_{i}^{2}+2j$.

Here $x$ is the position vector of $\mathbb{R}^{p}$, $x_{1},\dots,x_{p}$
are the cartesian coordinates of $\mathbb{R}^{p}$ and $\nu_{\alpha}$
is an orthonormal frame of the normal bundle of $M$. Moreover, the
functions $\left\langle x,v\right\rangle $, $\left|x\right|^{2}$
and $\sum_{i=1}^{j}x_{i}^{2}$ and their gradients are in $L_{f}^{2}\left(M\right)$.\end{lem}
\begin{proof}
First we prove Items (a), (b), (c) and (d). Item (a) follows from
Lemma \ref{lem:characterization}. We have
\begin{align*}
\Delta_{f}\left(\sum_{i=1}^{j}x_{i}^{2}\right) & =2\sum_{i=1}^{j}\Delta_{f}x_{i}\cdot x_{i}+2\sum_{i=1}^{j}\left|\nabla x_{i}\right|^{2}\\
 & =-\sum_{i=1}^{j}x_{i}^{2}+2\sum_{i=1}^{j}\left|\overline{\nabla}x_{i}\right|^{2}-2\sum_{i=1}^{j}\left|\left(\overline{\nabla}x_{i}\right)^{\perp}\right|^{2}.
\end{align*}
Here we consider the Riemannian connections $\nabla$ and $\overline{\nabla}$ of $M$ and $\mathbb{R}^{p}\times S_{\sqrt{2\left(k-1\right)}}^{k}$ respectively. On the other hand
\begin{align*}
\sum_{i=1}^{j}\left|\left(\overline{\nabla}x_{i}\right)^{\perp}\right|^{2} & =\sum_{i=1}^{j}\sum_{\alpha}\left\langle e_{i},\nu_{\alpha}\right\rangle ^{2}\\
 & =\sum_{\alpha}\left|\proj_{\mathbb{R}^{j}}\nu_{\alpha}\right|^{2},
\end{align*}
where $e_{1},\dots,e_{j}$ is the canonical basis of $\mathbb{R}^{j}$.
Thus 
\[
\Delta_{f}\left(\sum_{i=1}^{j}x_{i}^{2}\right)=-\sum_{i=1}^{j}x_{i}^{2}+2j-2\sum_{\alpha}\left|\proj_{\mathbb{R}^{j}}\nu_{\alpha}\right|^{2}.
\]
This formula proves Items (b), (c) and (d).

To complete the proof it suffices to prove that the functions $\left\langle x,v\right\rangle $
and $\sum_{i=1}^{j}x_{i}^{2}$ and their gradients are in $L_{f}^{2}\left(M\right)$.
Consider the open ball $B^{p}\left(o,r\right)$ in $\mathbb{R}^{p}$
of radius $r$ and center at the origin and the intrinsic open ball
$B_{r}^{M}$ in $M$ of radius $r$ and center at a fixed point of
$M$. Recall that $f\left(x,y\right)=\frac{\left|x\right|^{2}}{4}$.
We have
\begin{align*}
\int_{M}f^{2}e^{-f} & =\sum_{r=0}^{\infty}\int_{M\cap\left(\left(B^{p}\left(o,r+1\right)\setminus B^{p}\left(o,r\right)\right)\times S_{\sqrt{2\left(k-1\right)}}^{k}\right)}f^{2}e^{-f}\\
 & \leq C_{1}\sum_{r=0}^{\infty}\left(r+1\right)^{4}e^{-\frac{r^{2}}{4}}vol\left(M\cap\left(B^{p}\left(o,r+1\right)\times S_{\sqrt{2\left(k-1\right)}}^{k}\right)\right)\\
 & \leq C_{2}\sum_{r=0}^{\infty}\left(r+1\right)^{4}e^{-\frac{r^{2}}{4}}vol\left(B_{r+1}^{M}\right)\\
 & \leq C_{3}\sum_{r=0}^{\infty}\left(r+1\right)^{4}e^{-\frac{r^{2}}{4}}\left(r+1\right)^{\alpha}\\
 & <\infty,
\end{align*}
where $\alpha>0$. In the third inequality above we used
the fact that $M$ has polynomial volume growth (by Lemma \ref{lem:volume}).
Thus 
\[
\int_{M}\left|x\right|^{4}e^{-f}<\infty.
\]
Since $M$ has finite weighted volume (by Lemma \ref{lem:volume})
it follows that
\begin{align*}
\int_{M}\left(\sum_{i=1}^{j}x_{i}^{2}\right)e^{-f} & \leq\left(\int_{M}\left|x\right|^{4}e^{-f}\right)^{\frac{1}{2}}\left(\int_{M}e^{-f}\right)^{\frac{1}{2}}\\
 & <\infty.
\end{align*}
Clearly
\begin{align*}
\int_{M}\left\langle x,v\right\rangle ^{2}e^{-f} & \leq\left|v\right|^{2}\int_{M}\left|x\right|^{2}e^{-f}\\
 & <\infty,
\end{align*}
\begin{align*}
\int_{M}\left|\nabla\left\langle x,v\right\rangle \right|^{2}e^{-f} & \leq C_{4}\left|v\right|^{2}\int_{M}e^{-f}\\
 & <\infty,
\end{align*}
and
\begin{align*}
\int_{M}\left|\nabla\sum_{i=1}^{j}x_{i}^{2}\right|^{2}e^{-f} & \leq C_{5}\sum_{i=1}^{j}\int_{M}\left|x_{i}\right|^{2}\left|\nabla x_{i}\right|^{2}e^{-f}\\
 & \leq C_{5}\sum_{i=1}^{j}\int_{M}\left|x_{i}\right|^{2}e^{-f}\\
 & <\infty.
\end{align*}
\end{proof}
\begin{lem}
\label{lem:discrete} For a self-shrinker submanifold $M$ properly
immersed in a gradient shrinking Ricci soliton $\left(\overline{M},g,f\right)$
with $f$ convex, the drifted Laplacian of $M$ with
domain $L_{f}^{2}\left(M\right)$ has discrete spectrum.\end{lem}
\begin{proof}
Consider the Riemannian connections $\nabla$ and $\overline{\nabla}$ of $M$ and $\overline{M}$ respectively. We can assume that the Bakry-\'Emery-Ricci curvature of $\overline{M}$ satisfies $\overline{\ric_{f}}=\frac{1}{2}g$.
Consider the unitary isomorphism $U:L^{2}\left(M\right)\to L_{f}^{2}\left(M\right)$
given by $Uu=ue^{\frac{f}{2}}$ and the operator $L=\Delta+\frac{1}{2}\Delta f-\frac{1}{4}\left|\nabla f\right|^{2}$.
By direct calculation we have
\[
\Delta_{f}=ULU^{-1}.
\]
Thus the operator $\Delta_{f}$ with domain $L_{f}^{2}\left(M\right)$
has discrete spectrum if and only if the operator $L$ with domain
$L^{2}\left(M\right)$ has discrete spectrum. By spectral theory to
prove that $L$ has discrete spectrum it suffices to show that the
function $\frac{1}{4}\left|\nabla f\right|^{2}-\frac{1}{2}\Delta f$
is proper (e.g. \cite{RS} page 120). Since 
\[
\nabla\nabla f=\overline{\nabla}\overline{\nabla}f+\left\langle \overline{\nabla}f,A\left(\cdot,\cdot\right)\right\rangle ,
\]
it follows that 
\begin{align*}
\frac{1}{4}\left|\nabla f\right|^{2}-\frac{1}{2}\Delta f & =\frac{1}{4}\left|\overline{\nabla}f\right|^{2}-\frac{1}{4}\left|\left(\overline{\nabla}f\right)^{\perp}\right|^{2}-\frac{1}{2}\tr_{M}\overline{\nabla}\overline{\nabla}f-\frac{1}{2}\left\langle \overline{\nabla}f,\overrightarrow{H}\right\rangle \\
 & =\frac{1}{4}\left|\overline{\nabla}f\right|^{2}+\frac{1}{4}\left|\overrightarrow{H}\right|^{2}-\frac{1}{2}\tr_{M}\overline{\nabla}\overline{\nabla}f.
\end{align*}
This identity holds for any $f$-minimal submanifold of a smooth metric
measure space. By \cite{C} the scalar curvature of $\overline{M}$
satisfies $\overline{R}\geq0$. By \cite{Ham} $\left|\overline{\nabla}f\right|^{2}=f-\overline{R}$.
Since $\overline{M}$ is gradient shrinking Ricci soliton it follows
that $\overline{\Delta}f=\frac{\dim\overline{M}}{2}-\overline{R}$.
Take a local orthonormal frame $\left\{ \nu_{\alpha}\right\} $ in
the normal bundle of $M$. We have
\begin{align*}
\frac{1}{4}\left|\nabla f\right|^{2}-\frac{1}{2}\Delta f & =\frac{1}{4}\left(f-\overline{R}\right)+\frac{1}{4}\left|\overrightarrow{H}\right|^{2}-\frac{1}{2}\overline{\Delta}f+\frac{1}{2}\sum_{\alpha}\overline{\nabla}\overline{\nabla}f\left(\nu_{\alpha},\nu_{\alpha}\right)\\
 & =\frac{1}{4}\left|\overrightarrow{H}\right|^{2}+\frac{f}{4}+\frac{\overline{R}}{4}-\frac{\dim\overline{M}}{4}+\frac{1}{2}\sum_{\alpha}\overline{\nabla}\overline{\nabla}f\left(\nu_{\alpha},\nu_{\alpha}\right)\\
 & \geq\frac{f}{4}-\frac{\dim\overline{M}}{4}.
\end{align*}
Since $M$ is properly immersed in $\overline{M}$ and $f$ is a proper function on $\overline M$ (see \cite{MR2732975}) it follows that
the restriction of $f$ to $M$ is proper. This proves the result.
\end{proof}
By the divergence theorem, for any smooth function $u$ on a compact smooth metric measure space $\left(M,g,f\right)$ we have $\int_{M}\Delta_{f}u\cdot e^{-f}=0$.
But this is not generally true if $M$ is noncompact. It is an interesting problem to see when this formula holds because it has many geometric applications. We use the following lemma in the proofs of some results of the next section.
\begin{lem}
\label{lem:integralzero}Consider a complete smooth metric measure
space $\left(M,g,f\right)$ with finite weighted volume, i.e. $\int_{M}e^{-f}<\infty$.
Given a function $u$ on $M$ assume that $\nabla u$ is in $L_{f}^{2}\left(M\right)$.
If the drifted Laplacian $\Delta_{f}u$ does not change sign or it
is in $L_{f}^{1}\left(M\right)$, then 
\[
\int_{M}\Delta_{f}u\cdot e^{-f}=0.
\]
\end{lem}
\begin{proof}
Consider the open ball $B_{r}$ of radius $r$ and center at a fixed
point of $M$. Take a cutoff function $\phi$ on $M$ such that $\phi=1$
on $B_{r}$, $\phi=0$ on $M\setminus B_{r+1}$, $0\leq\phi\leq1$
and $\left|\nabla\phi\right|\leq C$ on $B_{r+1}\setminus B_{r}$.
Multiplying $\Delta_{f}u$ by $\phi$ and integrating by parts we obtain
\[
\int_{M}\Delta_{f}u\cdot\phi e^{-f}=-\int_{M}\left\langle \nabla u,\nabla\phi\right\rangle e^{-f}.
\]
By the monotone convergence theorem (if $\Delta_{f}u$ does not change
sign) or the dominated convergence theorem (if $\Delta_{f}u$ is in
$L_{f}^{1}\left(M\right)$), sending $r\to\infty$ we obtain
\[
\int_{M}\Delta_{f}u\cdot e^{-f}=-\lim_{r\to\infty}\int_{M}\left\langle \nabla u,\nabla\phi\right\rangle e^{-f}.
\]
Since $M$ has finite weighted volume and $\nabla u$ is in $L_{f}^{2}\left(M\right)$
it follows from Holder's inequality that the limit on the right hand
side is equal to zero.
\end{proof}
The lemma above can be generalized to $f$-parabolic manifolds. The
following result may be of independent interest.
\begin{lem}
Suppose a $f$-parabolic complete smooth metric measure space $\left(M,g,f\right)$. Given a function $u$ on $M$ assume that $\nabla u$
is in $L_{f}^{2}\left(M\right)$. If the drifted Laplacian $\Delta_{f}u$
does not change sign or it is in $L_{f}^{1}\left(M\right)$, then 
\[
\int_{M}\Delta_{f}u\cdot e^{-f}=0.
\]
\end{lem}
\begin{proof}
Take an exhausting sequence of open balls $\left\{ B_{i}\right\} $
and a sequence of $f$-harmonic functions $\left\{ \phi_{i}\right\} $
on $B_{i}\setminus B_{0}$ with boundary condition $\phi=1$ on $\partial B_{0}$
and $\phi=0$ on $B_{i}$. By the maximum principle we have
\[
\phi_{i}\leq\phi_{i+1},
\]
and by the divergence theorem we have
\[
\int_{B_{i}\setminus B_{0}}\left|\nabla\phi_{i}\right|^{2}e^{-f}=-\int_{\partial B_{0}}\frac{\partial\phi_{i}}{\partial\nu}e^{-f}.
\]
Since $M$ is $f$-parabolic it follows that $\phi_{i}\to1$ pointwise
and 
\[
\lim_{i\to\infty}\int_{B_{i}\setminus B_{0}}\left|\nabla\phi_{i}\right|^{2}e^{-f}=0,
\]
(e.g. \cite{L} chapter 20). Extend $\phi_{i}=1$ in $B_{0}$ and
$\phi_{i}=0$ in $M\setminus B_{i}$. Multiplying $\Delta_{f}u$ by
$\phi_{i}$ and integrating by parts we obtain
\[
\int_{M}\Delta_{f}u\cdot\phi_{i}e^{-f}=-\int_{M}\left\langle \nabla u,\nabla\phi_{i}\right\rangle e^{-f}.
\]
By the monotone convergence theorem (if $\Delta_{f}u$ does not change
sign) or the dominated convergence theorem (if $\Delta_{f}u$ is in
$L_{f}^{1}\left(M\right)$), sending $i\to\infty$ we obtain
\[
\int_{M}\Delta_{f}u\cdot e^{-f}=-\lim_{i\to\infty}\int_{M}\left\langle \nabla u,\nabla\phi_{i}\right\rangle e^{-f}.
\]
The limit on the right hand side is equal to zero since 
\[
\left|\int_{M}\left\langle \nabla u,\nabla\phi_{i}\right\rangle e^{-f}\right|\leq\left(\int_{M}\left|\nabla u\right|^{2}e^{-f}\right)^{\frac{1}{2}}\left(\int_{M}\left|\nabla\phi_{i}\right|^{2}e^{-f}\right)^{\frac{1}{2}}
\]
and $\nabla u$ is in $L_{f}^{2}\left(M\right)$.
\end{proof}

\section{Proofs of Theorem \ref{thm:selfshrinkereuclidean1} and Theorem \ref{thm:selfshrinkercylinder}}

We use the following lemma in the proofs of Theorem \ref{thm:selfshrinkereuclidean1}
and Corollary \ref{cor:selfshrinkereuclidean2}.
\begin{lem}
\label{lem:sphereselfshrinker}Given a vector $v$ in $\mathbb{R}^{k+1}$
and $r>0$ the hypersurface $S^{k}\left(v,r\right)\times\mathbb{R}^{n-k}$ in $\mathbb{R}^{n+1}$ is a self-shrinker if and only if it is the product $S^{k}\left(o,\sqrt{2k}\right)\times\mathbb{R}^{n-k}$.\end{lem}
\begin{proof}
Take the canonical basis $e_{1},\dots,e_{k+1}$ and the cartesian
coordinates $x_{1},\dots,x_{k+1}$ of $\mathbb{R}^{k+1}$. The hypersurface has normal vector $\nu=\frac{1}{r}\sum_{i=1}^{k+1}\left(x_{i}-v_{i}\right)e_{i}$
and mean curvature vector $\overrightarrow{H}=-\frac{k}{r}\nu$. Thus it is a self-shrinker if and only if
\[
\sum_{i=1}^{k+1}x_{i}\left(x_{i}-v_{i}\right)=2k.
\]
In particular, $S^{k}\left(o,\sqrt{2k} \right)\times\mathbb{R}^{n-k}$ is a
self-shrinker.

Now assume that the hypersurface is a self-shrinker. Combining the equation above with
\[
\sum_{i=1}^{k+1}\left(x_{i}-v_{i}\right)^{2}=r^{2}
\]
we see that $S^{k}\left(v,r\right)\times\mathbb{R}^{n-k}=S^{k}\left(o,s\right)\times\mathbb{R}^{n-k}$
where $s^{2}=4k-r^{2}+\sum_{i=1}^{k+1}v_{i}^{2}$. Thus $v=0$ and $r=\sqrt{2k}$.
\end{proof}
Now we prove Theorem \ref{thm:selfshrinkereuclidean1} and Corollary
\ref{cor:selfshrinkereuclidean2}.
\begin{proof}[Proof of Theorem \ref{thm:selfshrinkereuclidean1}]
The self-shrinker $M$ intercepts the sphere $S^{n}\left(tp,\sqrt{2n+t^{2}}\right)$
for any unit vector $p$ in $\mathbb{R}^{n+1}$ and $0\leq t<\infty$.
To see this consider the function $u$ on $M$ given by 
\[
u=2n+t^{2}-\left|x-tp\right|^{2}.
\]
Suppose that $M$ does not intercept the sphere $S^{n}\left(tp,\sqrt{2n+t^{2}}\right)$.
Then $u$ is positive or negative. We have
\[
u=u_{1}+2tu_{2},
\]
where $u_{1}=2n-\left|x\right|^{2}$ and $u_{2}=\left\langle x,p\right\rangle $.
By Lemma \ref{lem:equations} the function $u$ is a linear combination
of $L_{f}^{2}\left(M\right)$ eigenfunctions of the drifted Laplacian
of $M$ with nonzero eigenvalues. Since $M$ has finite weighted volume
(by Lemma \ref{lem:volume}) it follows that $u$ is orthogonal to
the constant function with respect to the $L_{f}^{2}\left(M\right)$
inner product i.e. $\int_{M}ue^{-f}=0$, a contradiction.

The self-shrinker $M$ intercepts the halfspace $\left\langle x,p\right\rangle =0$
in $\mathbb{R}^{n+1}$ for any unit vector $p$ in $\mathbb{R}^{n+1}$.
To see this consider the function $v$ on $M$ given by 
\[
v=\left\langle x,p\right\rangle .
\]
Suppose that $M$ does not intercept the halfspace $\left\langle x,p\right\rangle =0$.
Then $v$ is positive or negative. By Lemma \ref{lem:equations} this
function is a $L_{f}^{2}\left(M\right)$ eigenfunction of the drifted
Laplacian of $M$ with nonzero eigenvalue. As before $\int_{M}ve^{-f}=0$,
a contradiction.

By the facts above the self-shrinker $M$ intercepts all members of the collection
$\mathcal{C}$.

Now assume that $M$ lies in the closed set $\overline{B^{n+1}\left(tp,\sqrt{2n+t^{2}}\right)}$.
For $0\leq t<\infty$ the function $u$ above is non-negative with
$\int_{M}ue^{-f}=0$, so $u$ is identically zero and $M$ is the
sphere $S^{n}\left(tp,\sqrt{2n+t^{2}}\right)$, which by Lemma \ref{lem:sphereselfshrinker}
is the sphere $S^{n}\left(o,\sqrt{2n}\right)$. For $t=\infty$ the
function $v$ above is non-negative with $\int_{M}ve^{-f}=0$, so
$v$ is identically zero and $M$ is the hyperplane $\left\langle x,p\right\rangle =0$.
\end{proof}

\begin{proof}[Proof of Corollary \ref{cor:selfshrinkereuclidean2}]
Assume that the Gaussian image of the self-shrinker $M$ lies in
a closed semisphere. Given a vector $v$ in $\mathbb{R}^{n+1}$ consider
the function $u$ on $M$ given by 
\[
u=\left\langle \nu,v\right\rangle .
\]
Then $u$ is non-negative for some vector $v$ in $\mathbb{R}^{n+1}$.
By Lemma 5.5 in \cite{CM} (or Proposition 2 in \cite{CMZ2}) we have
\[
\Delta_{f}u+\left|A\right|^{2}u=0.
\]
If $u$ is identically zero then $M$ is a hyperplane. If $u$ is
not identically zero then the first eigenvalue of the operator $\Delta_{f}+\left|A\right|^{2}$
is non-negative, so 
\[
\int_{M}\left|A\right|^{2}\phi^{2}e^{-f}\leq\int_{M}\left|\nabla\phi\right|^{2}e^{-f}
\]
for all smooth functions $\phi$ on $M$ with compact support. Consider
the intrinsic open ball $B_{r}$ on $M$ of radius $r$ and center
at a fixed point of $M$. Take a cutoff function $\phi$ on $M$ such
that $\phi=1$ on $B_{r}$, $\phi=0$ on $M\setminus B_{r+1}$, $0\leq\phi\leq1$
and $\left|\nabla\phi\right|\leq C$ on $B_{r+1}\setminus B_{r}$.
Since $M$ has finite weighted volume (by Lemma \ref{lem:volume}), by the monotone convergence theorem sending $r\to\infty$ we see that that $A$ is identically zero and $M$ is a hyperplane.

Suppose that the self-shrinker $M$ lies inside the closed product
\[
\overline{B^{k+1}\left(v,\sqrt{2k+\left|v\right|^{2}}\right)}\times\mathbb{R}^{n-k}.
\]
Consider the function $u$ on $M$ given by 
\[
u=\sum_{i=1}^{k+1}\left(x_{i}-v_{i}\right)^{2}-2k-\sum_{i=1}^{k+1}v_{i}^{2}.
\]
Then $u$ is non-positive. We have
\[
u=u_{1}-2u_{2},
\]
where $u_{1}=\sum_{i=1}^{k+1}x_{i}^{2}-2k$ and $u_{2}=\sum_{i=1}^{k+1}x_{i}v_{i}$.
By Lemma \ref{lem:equations} the functions $u_{1}$ and $u_{2}$
and their gradients are in $L_{f}^{2}\left(M\right)$ and satisfy
$\Delta_{f}u_{1}\geq-u_{1}$ and $\Delta_{f}u_{2}=-\frac{1}{2}u_{2}$.
Thus 
\[
0\geq u\geq-\Delta_{f}u_{1}+4\Delta_{f}u_{2}.
\]
To prove that $u$ is identically zero it suffices to show that $\int_{M}\Delta_{f}u_{1}\cdot e^{-f}=0$
and $\int_{M}\Delta_{f}u_{2}\cdot e^{-f}=0$. The function $u_{2}$
is a $L_{f}^{2}\left(M\right)$ eigenfunction of the drifted Laplacian
of $M$ with nonzero eigenvalue. Since $M$ has finite weighted volume
(by Lemma \ref{lem:volume}) it follows that $u_{2}$ is orthogonal
to the constant function with respect to the $L_{f}^{2}\left(M\right)$
inner product i.e. $\int_{M}u_{2}e^{-f}=0$, so $\int_{M}\Delta_{f}u_{2}\cdot e^{-f}=0$.
By Lemma \ref{lem:equations} we have
\[
\left|\Delta_{f}u_{1}\right|\leq\sum_{i=1}^{k+1}x_{i}^{2}+2\left(k+1\right).
\]
By Lemma \ref{lem:volume} and Lemma \ref{lem:equations} the functions
in the right hand side are in $L_{f}^{1}\left(M\right)$. It follows
from Lemma \ref{lem:integralzero} that $\int_{M}\Delta_{f}u_{1}\cdot e^{-f}=0$.
Thus $u$ is identically zero and $M$ is the product $S^{k}\left(v,\sqrt{\left|v\right|^{2}+2k}\right)\times\mathbb{R}^{n-k}$,
which by Lemma \ref{lem:sphereselfshrinker} is the product $S^{k}\left(o,\sqrt{2k}\right)\times\mathbb{R}^{n-k}$.

Suppose that the self-shrinker $M$ lies outside the closed product
\[
\overline{B^{k+1}\left(v,\sqrt{\left|v\right|^{2}+2\left(k+1\right)}\right)}\times\mathbb{R}^{n-k}.
\]
Consider the function $u$ on $M$ given by 
\[
u=\sum_{i=1}^{k+1}\left(x_{i}-v_{i}\right)^{2}-2\left(k+1\right)-\sum_{i=1}^{k+1}v_{i}^{2}.
\]
Then $u$ is non-negative. We have
\[
u=u_{1}-2u_{2},
\]
where $u_{1}=\sum_{i=1}^{k+1}x_{i}^{2}-2\left(k+1\right)$ and $u_{2}=\sum_{i=1}^{k+1}x_{i}v_{i}$.
By Lemma \ref{lem:equations} the functions $u_{1}$ and $u_{2}$
and their gradients are in $L_{f}^{2}\left(M\right)$ and satisfy
$\Delta_{f}u_{1}\leq-u_{1}$ and $\Delta_{f}u_{2}=-\frac{1}{2}u_{2}$.
Thus 
\[
0\leq u\leq-\Delta_{f}u_{1}+4\Delta_{f}u_{2}.
\]
As before $u$ is identically zero and $M$ is the product $S^{k}\left(v,\sqrt{\left|v\right|^{2}+2\left(k+1\right)}\right)\times\mathbb{R}^{n-k}$.
By Lemma \ref{lem:sphereselfshrinker} this product is not a self-shrinker,
a contradiction.

Suppose that there is a hyperplane $\left\langle x,v\right\rangle =0$
in $\mathbb{R}^{n+1}$ separating the self-shrinker $M$ into two
parts with one part lying inside the closed ball $\overline{B^{n+1}\left(o,\sqrt{2n}\right)}$
and the other part lying outside the open ball $B^{n+1}\left(o,\sqrt{2n}\right)$.
Consider the function $u$ on $M$ given by
\[
u=\left(2n-\left|x\right|^{2}\right)\left\langle x,v\right\rangle .
\]
Then $u$ is non-negative or non-positive. We have
\[
u=u_{1}u_{2},
\]
where $u_{1}=2n-\left|x\right|^{2}$ and $u_{2}=\left\langle x,v\right\rangle $.
By Lemma \ref{lem:equations} the functions $u_{1}$ and $u_{2}$
are $L_{f}^{2}\left(M\right)$ eigenfunctions of the drifted Laplacian
of $M$ with distinct eigenvalues, so $u_{1}$ is orthogonal to $u_{2}$
with respect to the $L_{f}^{2}\left(M\right)$ inner product i.e.
$\int_{M}u_{1}u_{2}e^{-f}=0$. Thus $u$ is identically zero and $M$
is the sphere $S^{n}\left(o,\sqrt{2n}\right)$.
\end{proof}
By the proof of Lemma \ref{lem:sphereselfshrinker} we have the following
result.
\begin{lem}
\label{lem:sphereselfshrinker2}Given $r>0$ the hypersurface $S_{r}^{n-k}\times S_{\sqrt{2\left(k-1\right)}}^{k}$ in the gradient shrinking Ricci soliton
$\mathbb{R}^{n+1-k}\times S_{\sqrt{2\left(k-1\right)}}^{k}$ is a self-shrinker if and only if it is the product $S_{\sqrt{2\left(n-k\right)}}^{n-k}\times S_{\sqrt{2\left(k-1\right)}}^{k}$.
\end{lem}
Now we prove Theorem \ref{thm:selfshrinkercylinder}.
\begin{proof}[Proof of Theorem \ref{thm:selfshrinkercylinder}]
Suppose that the self-shrinker $M$ lies inside the closed product
$\overline{B_{\sqrt{2\left(n-k\right)}}^{n+1-k}}\times S_{\sqrt{2\left(k-1\right)}}^{k}$.
Consider the function $u$ on $M$ given by 
\[
u=2\left(n-k\right)-\left|x\right|^{2},
\]
where $x$ is the position vector in $\mathbb{R}^{n+1-k}$. Then $u$ is non-negative. By Lemma \ref{lem:equations} the function
$u$ and its gradient are in $L_{f}^{2}\left(M\right)$ and satisfies
\[
0\leq u\leq-\Delta_{f}u.
\]
Since $\Delta_{f}u$ does not change sign, by Lemma \ref{lem:integralzero} we have $\int_{M}\Delta_{f}u\cdot e^{-f}=0$.
Thus $u$ is identically zero and $M$ is the product $S_{\sqrt{2\left(n-k\right)}}^{n-k}\times S_{\sqrt{2\left(k-1\right)}}^{k}$.

Suppose that the self-shrinker $M$ lies outside the product $\overline{B_{\sqrt{2\left(n+1-k\right)}}^{n+1-k}}\times S_{\sqrt{2\left(k-1\right)}}^{k}$.
Consider the function $u$ on $M$ given by 
\[
u=\left|x\right|^{2}-2\left(n+1-k\right).
\]
Then $u$ is non-negative. By Lemma \ref{lem:equations} the function
$u$ and its gradient are in $L_{f}^{2}\left(M\right)$ and satisfies
\[
0\leq u\leq-\Delta_{f}u.
\]
As before $u$ is identically zero and $M$ is the product $S_{\sqrt{2\left(n+1-k\right)}}^{n-k}\times S_{\sqrt{2\left(k-1\right)}}^{k}$.
By Lemma \ref{lem:sphereselfshrinker2} this product is not a self-shrinker,
a contradiction.

Suppose that the self-shrinker $M$ lies in the product of the closed
halfspace satisfying $\sum_{i=1}^{n+1-k}x_{i}v_{i}\geq0$ in $\mathbb{R}^{n+1-k}$
with the sphere $S_{\sqrt{2\left(k-1\right)}}^{k}$ for some vector
$v$ in $\mathbb{R}^{n+1-k}$. Consider the function $u$ on $M$
given by 
\[
u=\sum_{i=1}^{n+1-k}x_{i}v_{i}.
\]
Then $u$ is non-negative. By Lemma \ref{lem:equations} this function
is a $L_{f}^{2}\left(M\right)$ eigenfunction of the drifted Laplacian
of $M$ with nonzero eigenvalue. Since $M$ has finite weighed volume
(by Lemma \ref{lem:volume}) it follows that $u$ is orthogonal to
the constant function with respect to the $L_{f}^{2}\left(M\right)$
inner product i.e. $\int_{M}ue^{-f}=0$. Thus $u$ is identically
zero and $M$ is the product of the hyperplane $\sum_{i=1}^{n+1-k}x_{i}v_{i}=0$
with the sphere $S_{\sqrt{2\left(k-1\right)}}^{k}$.\end{proof}

\end{document}